\documentclass[12pt]{article}       

\usepackage{amsmath, amssymb, amsfonts, amsopn, amsthm}
\usepackage{graphicx}
\usepackage{subfigure}
\usepackage[percent]{overpic}
\newcommand{\Z}{\mathbb{Z}}
\newcommand{\C}{\mathbb{C}}
\newcommand{\Q}{\mathbb{Q}}
\newcommand{\rp}{\mathbb{R}\text{P}}

\newcommand{\ri}{\Omega_1}
\newcommand{\rii}{\Omega_2}
\newcommand{\riii}{\Omega_3}
\newcommand{\riv}{\Omega_4}
\newcommand{\rv}{\Omega_5}
\newcommand{\rvi}{\Omega_6}
\newcommand{\rvii}{\Omega_7}
\newcommand{\rviii}{\Omega_8}
\newcommand{\rix}{\Omega_9}

\newcommand{\rpq}{\Omega_{\alpha,\beta}}
\newcommand{\s}[1]{\ensuremath{\bold{S}^{#1}}}

\allowdisplaybreaks

\begin{document}

\title{On the Seifert fibered space link group}

\author{Bo\v stjan Gabrov\v sek \and Enrico Manfredi}




\theoremstyle{plain}
\newtheorem{theorem}{Theorem}[section]
\newtheorem{lemma}[theorem]{Lemma}
\newtheorem{corollary}[theorem]{Corollary}
\newtheorem{definition}[theorem]{Definition}

\theoremstyle{definition}
\newtheorem{example}[theorem]{Example} 
\newtheorem*{remark}{Remark}

\maketitle

\begin{abstract}
We introduce generalized arrow diagrams and generalized Reidemeister moves for diagrams of links in Seifert fibered spaces. We give a presentation of the fundamental group of the link complement. As a corollary we are able to compute the first homology group of the complement and the twisted Alexander polynomials of the link.
\noindent
{{\it Mathematics Subject
Classification 2010:} Primary 57M27; Secondary 57M05.}\\
\end{abstract}

\section{Introduction}

A Seifert fibered space $M$ can be constructed from an $S^1$-bundle $\pi: E \rightarrow F$ over a surface $F$, with some possible disjoint rational $(\alpha, \beta)$-surgeries performed along the fibers $\pi^{-1}(x)$, $x\in F$. Here the coefficients $\alpha$ and $\beta$ are two coprime integers where $0 \leq \beta < \alpha$. The fibers along which the surgeries are performed are called exceptional fibers, all other fibers are called ordinary. If the base space $F$ is compact then $M$ is compact and the number of exceptional fibers is finite~\cite{S}.

Links in Seifert fibered spaces are embeddings of several disjoint copies of circles $\s1$. As for links in the $3$-sphere, we may want to find suitable ways to represent them and we may want to find invariants to distinguish them. 

One of the main reason to be interested in links in Seifert fibered spaces is the computation of Skein Modules of these particular $3$-manifolds. Skein modules are not only important invariants of $3$-manifolds, but are also useful for topological quantum field theories.


Links in $S^1$-bundles are best described by arrow diagrams introduced by Mroczkowski in \cite{Mr}. Since Seifert fibered spaces resemble $S^1$-bundles, we use these types of diagrams as in \cite{Mr2}.
In Section~\ref{linkmoves} we present a generalized planar arrow diagram for a link in a Seifert fibered space, together with a list of ten generalized Reidemeister moves that satisfy the request that two link diagrams represent the same link up to ambient isotopy if and only if there exists a finite sequence of generalized Reidemeister moves between the two diagrams. 

On the other hand we can consider studying links only up to difffeomorphism, i.e. we say that two links $L_1$ and $L_2$ are diffeo-equivalent if and only if there exists a self-diffeomorphism $h: M \rightarrow M$ such that $h(L_1) = L_2$. This condition is weaker than ambient isotopy and we explore it in Section~\ref{linkgroup} by providing a method of calculating the group of a link which is able to detect diffeo-inequivalent links.

In Section~\ref{linkhomology}, the link group is abelianized in order to obtain the first homology group of the complement. In some cases it is possible to recover it directly from the diagram, using the homology class of the link components. With some assumption on the manifold and on the link, it is possible to understand the behaviour of the rank and of the torsion of the homology.

In Section~\ref{linkTAP} we exploit the group presentation and the homology characterization to compute through Fox calculus a class of twisted Alexander polynomials, corresponding to a particular 1-dimensional representation of the link group. This particular class contains the usual Alexander polynomial and a family of twisted polynomials that is able to keep track of the torsion of the group; for example the polynomials of this family becomes zero on local links. Moreover, as expected, the polynomials split under connected sum of links. 
At last, in Section~\ref{linkexample}, an example illustrates all the machinery developed.


\section{Diagrams of links in Seifert fibered spaces}\label{linkmoves}


In order to define diagrams of knots inside Seifert fibered spaces, we need to explicit their construction.

Let $F$ be a compact surface and let $G$ be the fundamental polygon of $F \approx G/_\sim$ where $\sim$ is the equivalence relation that identifies the points on the boundary of the fundamental polygon as depicted on Figure \ref{fig:polygons}. Denote the edges by $a_{i}$ or $b_{i}$. If $F$ is an orientable genus $g>0$ surface, we have the identification $a_1 b_1 a_1^{-1} b_1^{-1} a_2 b_2 a_2^{-1} b_2^{-1} \cdots a_g b_g a_g^{-1} b_g^{-1}$ of the edges (Figure~\ref{fig:poly1}), if $F$ is a sphere we have the identification $a_1 a_1^{-1}$ (Figure~\ref{fig:poly2}), and if $F$ is a non-orientable genus $g$ surface, we have the identification $a_1^2 a_2^2 \cdots a_g^2$ (Figure~\ref{fig:poly3}).

\begin{figure}[htb]
\centering
\subfigure[$T^2 \# T^2$]{\begin{overpic}[page=3]{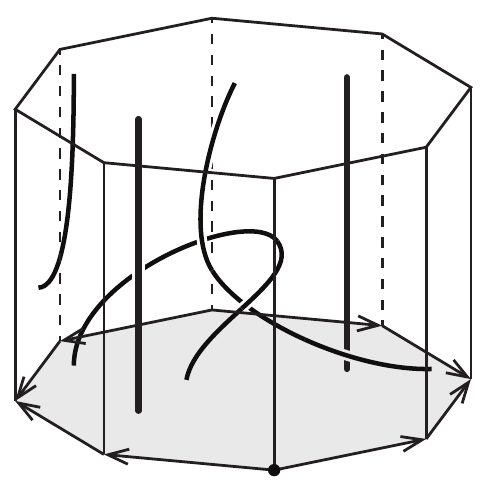}
\put(14,13){$b_2$}\put(-2,46){$a_2$}\put(10,82){$b_2$}\put(47,95){$a_2$}
\put(79,82){$b_1$}\put(94,46){$a_1$}\put(79,13){$b_1$}\put(45,0){$a_1$}
\end{overpic}\label{fig:poly1}}\hspace{.5cm}
\subfigure[$S^2$]{\begin{overpic}[page=38]{seif}
\put(47,3){$a_1$}\put(47,91){$a_1$}
\end{overpic}\label{fig:poly2}}\hspace{.5cm}
\subfigure[$K \# \rp^2$]{\begin{overpic}[page=4]{seif}
\put(9,26){$a_3$}\put(9,70){$a_3$}\put(47,89){$a_2$}\put(80,70){$a_2$}\put(80,26){$a_1$}\put(47,6){$a_1$}
\end{overpic}\label{fig:poly3}}
\caption{Fundamental polygons.}
\label{fig:polygons}
\end{figure}
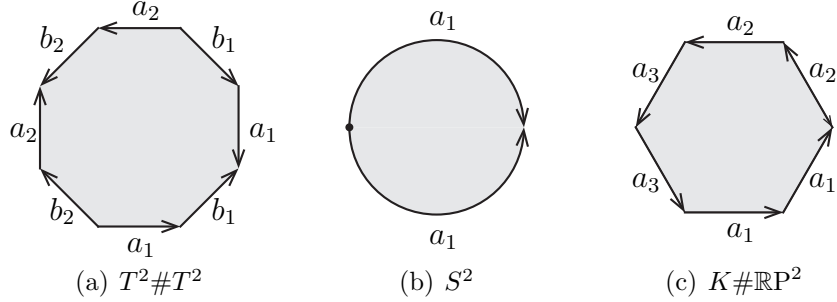

An arbitrary compact Seifert fibered space $M$ can be constructed from $G$ as follows.
Take $G \times [0,1]$ and by glueing each $\{x\}\times \{  0 \}$ to $\{x\}\times \{  1 \}$, $x\in G$, we get the trivial circle bundle $G \times S^1$. Since $G$ is a disk, we can orient all the fibers $\{x\} \times S^1$ coherently. If two oriented edges $a_i$ and $a_i'$ are identified in $G$, in order to get $F$, we can identify the tori $a_i \times S^1$ and $a_i' \times S^1$ in two essentially different ways: $a_i \times S^1$ is glued to $a_i' \times S^1$ by identity or by a reflection on the $S^1$ component. According to this, we assign to each edge $a_i$ the sign $\gamma_i = \pm 1$ and to each edge $b_i$ the sign  $\delta_i = \pm 1$. In both cases $+1$ is chosen if the identification is made by the identity and $-1$ otherwise.

After the above identifications the resulting space is just a compact $S^1$-bundle over $F$.
We can get an arbitrary Seifert fibered space by performing $(\alpha_i, \beta_i)$-surgeries along a $k$ disjoint fibers $\{x_i\}\times S^1$, $x_i \in F$ for $0 \leq i \leq k$, where $\alpha_i$ and $\beta_i$ are again two coprime integers, $0 \leq \alpha_i < \beta_i$\footnote{If we refer to Seifert's original paper~\cite{S}, Seifert fibered spaces are usually denoted by $S(O,o,g | b; \alpha_1, \beta_1; \alpha_2, \beta_2; \ldots; \alpha_k, \beta_k)$; in our case we use the assumption $b=0$, which is not restrictive, since, if we add a surgery with coefficients $(b,1)$, we get the same result.}.

We remark that for an orientable Seifert fibered space $M$, the base space does not need to be oriented, but the values $\gamma_i$ (and $\delta_i$) are determined. In the case $F$ is orientable,  $\gamma_i = \delta_i = +1$ and in the case $F$ in non-orientable $\gamma_i = -1$ for all $i$.


If a link $L$ lies in a thickened surface $F\times I$, the diagram of $L$ is just the regular projection of $L$ onto $F \approx F \times \{0\}$ along with the information of under- and overcrossings with respect to the projection. If we present $F$ by its fundamental polygon, we obtain a set of curves in $G$.
By glueing $F\times \{0\}$ to $F\times \{1\}$ we get a $S^1$-bundle $M$ over $F$. We now map a link $L \subset M$ to $F$ by the induced projection $p$ and keep track of where $L$ passes the section $F\times \{0\}$ by decorating the diagram with an arrow at the passage point, where the arrow indicates the direction we should travel along $L$ in order to cross $F \times \{ 1 \}$ and emerge from $F \times \{ 0 \}$ (Figure \ref{fig:arrows}). Let us remark that this induced projection agrees with the fibration $\pi$ in the case we do not have any exceptional fibers in $M$.

\begin{figure}[htb]
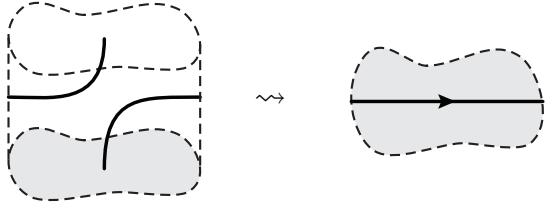

\centering
\begin{overpic}[page=5]{seif}\end{overpic}\raisebox{40pt}{$\;\;\leadsto\;\;$}\
\begin{overpic}[page=6]{seif}\end{overpic}
\caption{Constructing an arrow diagram.}
\label{fig:arrows}
\end{figure}

Since the projection maps an exceptional fiber to a point in the base space, it is enough to specify the image of each exceptional fiber in $G$, which is done by placing a point on $G$ decorated by the surgery coefficient $(\alpha_j, \beta_j)$ of the fiber (Figure \ref{fig:diagram}).

\begin{figure}[htb]
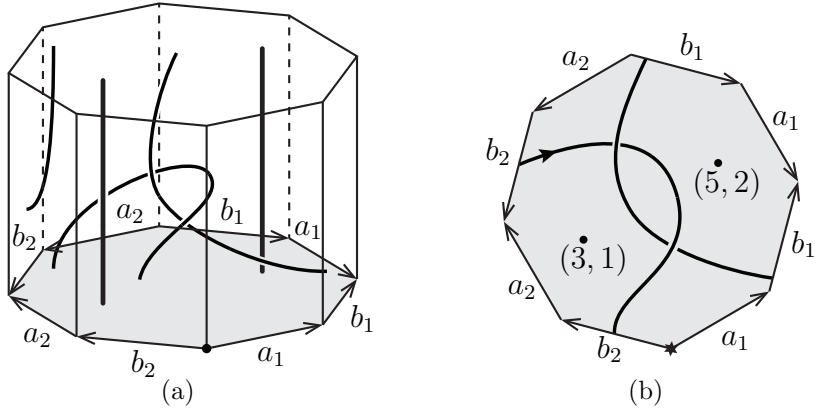

\centering
\subfigure[]{\begin{overpic}[page=1]{seif}
\put(36,-3){$b_2$}\put(7,5){$a_2$}\put(4,31){$b_2$}\put(32,38.5){$a_2$}   
\put(60,38){$b_1$}\put(80,34){$a_1$}\put(95,9){$b_1$}\put(70,0){$a_1$} 
\end{overpic}\label{fig:diag1}}\hspace{1.5cm}
\subfigure[]{\begin{overpic}[page=2]{seif}
\put(33,1){$b_2$}\put(5,21){$a_2$}\put(-2,63){$b_2$}\put(22,93){$a_2$}   
\put(59,96){$b_1$}\put(88,74){$a_1$}\put(94,34){$b_1$}\put(72,5){$a_1$} 
\put(21,29){$(3,1)$}\put(63,53){$(5,2)$}\end{overpic}\label{fig:diag2}}
\caption{A diagram of a link (a) and its projection (b) in a Seifert fibered space with an orientable genus 2 base space and $\gamma_1 = -1$, $\delta_1 = \gamma_2 = \delta_2 = +1$ and two exceptional points with coefficients $(3,1)$ and $(5,2)$.}
\label{fig:diagram}
\end{figure}

We complete this section by providing a list of generalized Reidemeister moves associated to the arrow diagrams. In the interior of $G$ and outside of exceptional fibers we have three classical moves Reidemeister  moves $\ri - \riii$  and two moves $\riv$ and $\rv$ involving arrows~\cite{Mr,Mr2,MD}(Figure \ref{fig:reid1}).

\begin{figure}[htb]
\centering
\subfigure[$\ri$]{\begin{overpic}[page=9]{seif}\end{overpic}\raisebox{10pt}{$\;\;\longleftrightarrow\;\;$}\begin{overpic}[page=10]{seif}\end{overpic}}
\hspace{1cm}
\subfigure[$\rii$]{\begin{overpic}[page=11]{seif}\end{overpic}\raisebox{10pt}{$\;\;\longleftrightarrow\;\;$}\begin{overpic}[page=12]{seif}\end{overpic}}
\hspace{1cm}
\subfigure[$\riii$]{\begin{overpic}[page=13]{seif}\end{overpic}\raisebox{10pt}{$\;\;\longleftrightarrow\;\;$}\begin{overpic}[page=14]{seif}\end{overpic}}
\subfigure[$\riv$]{\begin{overpic}[page=15]{seif}\end{overpic}\raisebox{12pt}{$\;\;\longleftrightarrow\;\;$}\begin{overpic}[page=16]{seif}\end{overpic}\raisebox{12pt}{$\;\;\longleftrightarrow\;\;$}\begin{overpic}[page=17]{seif}\end{overpic}}
\hspace{1cm}
\subfigure[$\rv$]{\begin{overpic}[page=18]{seif}\end{overpic}\raisebox{12pt}{$\;\;\longleftrightarrow\;\;$}\begin{overpic}[page=19]{seif}\end{overpic}}
\hspace{1cm}
\caption{Classical Reidemesiter moves $\ri$ -- $\riii$ and two ``arrow'' moves $\riv$ and $\rv$.}
\label{fig:reid1}
\end{figure}

Generalized Reidemeister moves $\rvi$ -- $\rix$ act across edges in $G$ (Figure~\ref{fig:reid2}). The move $\rvi$ corresponds to pushing an arc over an edge~\cite{CMM,Mr,Mr2}, $\rvii$ corresponds to pushing a crossing over the edge and comes in two variants: $\rvii^O$ in the case $F$ is orientable and $\rvii^N$ in the case $F$ is non-orientable~\cite{Mr,Mr2,CMM}.  The move $\rvii^\pm$ corresponds to pushing an arrow over an edge, where the sign denotes the sign associated to the edge~\cite{Mr,Mr2}. The move $\rix$ corresponds to pushing an arc over the base point of $G$ and is similar to the move $R_7$ in~\cite{CMM}. The $\rix$ move comes in three flavours: if $G$ is a orientable genus $g>0$ surface we have $\rix^O$, if $G$ is the 2-sphere we have $\rix^S$, and if $G$ is a non-orientable surface we have the move $\rix^N$. Figure~\ref{fig:vis} shows the geometrical interpretation of $\rix$ in the case of a double torus.

\begin{figure}[htb]
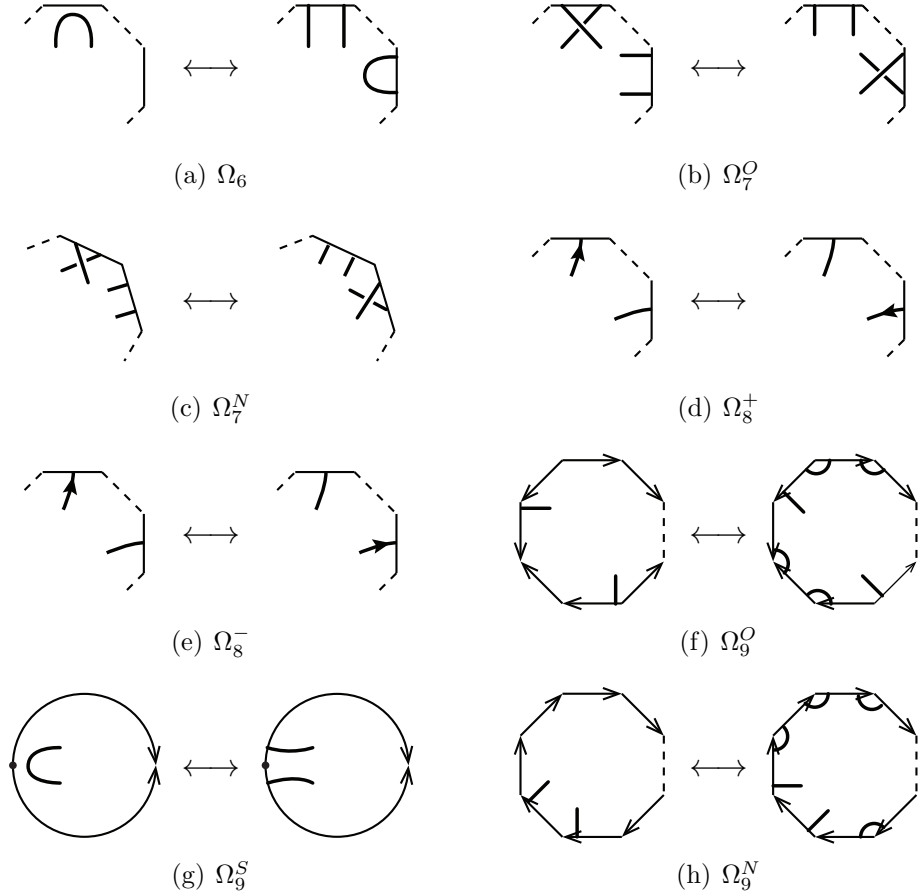

\centering
\subfigure[$\rvi$]{\begin{overpic}[page=20]{seif}\end{overpic}\raisebox{26pt}{$\;\;\longleftrightarrow\;\;$}\begin{overpic}[page=21]{seif}\end{overpic}}
\hspace{1cm}
\subfigure[$\rvii^O$]{\begin{overpic}[page=22]{seif}\end{overpic}\raisebox{26pt}{$\;\;\longleftrightarrow\;\;$}\begin{overpic}[page=23]{seif}\end{overpic}}
\subfigure[$\rvii^N$]{\begin{overpic}[page=24]{seif}\end{overpic}\raisebox{26pt}{$\;\;\longleftrightarrow\;\;$}\begin{overpic}[page=25]{seif}\end{overpic}}
\hspace{1cm}
\subfigure[$\rviii^+$]{\begin{overpic}[page=30]{seif}\end{overpic}\raisebox{26pt}{$\;\;\longleftrightarrow\;\;$}\begin{overpic}[page=31]{seif}\end{overpic}}
\subfigure[$\rviii^-$]{\begin{overpic}[page=32]{seif}\end{overpic}\raisebox{26pt}{$\;\;\longleftrightarrow\;\;$}\begin{overpic}[page=33]{seif}\end{overpic}}
\hspace{1cm}
\subfigure[$\rix^O$]{\begin{overpic}[page=26]{seif}\end{overpic}\raisebox{26pt}{$\;\;\longleftrightarrow\;\;$}\begin{overpic}[page=27]{seif}\end{overpic}}
\subfigure[$\rix^S$]{\begin{overpic}[page=40]{seif}\end{overpic}\raisebox{26pt}{$\;\;\longleftrightarrow\;\;$}\begin{overpic}[page=41]{seif}\end{overpic}}
\hspace{1cm}
\subfigure[$\rix^N$]{\begin{overpic}[page=28]{seif}\end{overpic}\raisebox{26pt}{$\;\;\longleftrightarrow\;\;$}\begin{overpic}[page=29]{seif}\end{overpic}}
\caption{Additional Reidemeister moves.}
\label{fig:reid2}
\end{figure}

\begin{figure}[htb]
\centering
\begin{overpic}[page=37]{seif}
\put(21,35){$a_1$}\put(71,35){$a_2$}
\put(20,6){$b_1$}\put(76,6){$b_2$}
\end{overpic}\raisebox{12pt}{$\;\;\longleftrightarrow\;\;$}\begin{overpic}[page=39]{seif}
\put(21,35){$a_1$}\put(71,35){$a_2$}
\put(20,6){$b_1$}\put(76,6){$b_2$}
\end{overpic}
\caption{Visualization of the move $\rix^O$.}
\label{fig:vis}
\end{figure}

Diagrams with exceptional points are equipped with an additional ``slide'' move (also known as the band move) $\rpq$ that corresponds to sliding an arc over the exceptional point in a diagram, i.e. sliding an arc of a link over the meridional disk of the solid torus which is attached when performing the $(\alpha,\beta)$-surgery. The move consists of going $\beta$ times around the exceptional point and adding $\alpha$ arrows uniformly on every $2 \pi \beta / \alpha$ angle as shown on Figure~\ref{fig:slide}, see~\cite{Mr2} for more details.

\begin{figure}[htb]
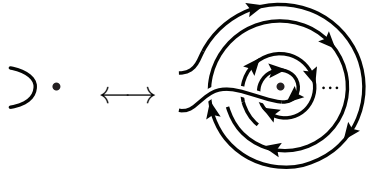

\centering
\begin{overpic}[page=34]{seif}\end{overpic}\raisebox{25pt}{$\;\;\longleftrightarrow\;\;$}\begin{overpic}[page=35]{seif}\end{overpic}
\caption{The slide move $\rpq$ over an exceptional point.}
\label{fig:slide}
\end{figure}


\section{The link group}\label{linkgroup}

If the Seifert fibered space $M$ is described as an $S^{1}$-bundle over the surface $F$ with possible rational surgeries and a link $L$ in $M$ is described 
by a generalized arrow diagram as in the previous section,
then we can find a presentation for the group of the link. Before the description, we recall a standard presentation for the group of the  Seifert fibered space $M$ itself, following \cite{O}.

\paragraph{The group of the Seifert fibered space}

If $F$ is orientable, 
the fundamental group $\pi_1(M)$ has the following presentation:

\begin{align}\label{gruppoFor}
\pi_1(M) = \;& \langle a_1, b_1, \ldots, a_g, b_g, q_1, \ldots , q_k, h \; | \; 
[a_1, b_1] \cdots [a_g,b_g]  q_1 \cdots q_k = 1, \nonumber \\
& a_i h a_{i}^{-1} h^{-\gamma_{i}}=1,
b_i h b_{i}^{-1}  h^{-\delta_{i}} =1,
\forall i \in \{ 1, \ldots, g \}, \nonumber \\
& [q_i, h] =1,
q_i^{\alpha_i} h^{\beta_i} =1,
\forall i \in \{ 1, \ldots, k \}\rangle.
\end{align}

For $F$ non-orientable the presentation is as follows:

\begin{align}\label{gruppoFnor}
\pi_1(M) = \;& \langle a_1, \ldots, a_g, q_1, \ldots , q_k, h \; | \;
a_1^2 \cdots a_g^2 q_1 \cdots q_k = 1, \nonumber \\
& a_i h a_{i}^{-1} h^{-\gamma_{i}}=1,
\forall i \in \{ 1, \ldots, g \} , \nonumber \\
& [q_i, h] =1,
q_i^{\alpha_i} h^{\beta_i} =1,
\forall i \in \{ 1, \ldots, k \}\rangle.
\end{align}

Where $g$ is the genus of the underlying basis $F$, the integer $k$ is the number of surgeries performed with indexes $(\alpha_{i},\beta_{i})$ along the $S^{1}$-fiber $q_{i}$, for every $i= 1, \ldots, k$, the generators $\{a_i , b_i \}_{i \in \{ 1, \ldots, g \} }$ (resp. $\{a_i \}_{i \in \{ 1, \ldots, g \} }$ for the unorientable case) are the standard generators of $\pi_1 (F)$ while $\gamma_{i}$ and $\delta_{i}= \pm 1$ are chosen according to the glueing orientation of the lateral surface of $a_{i}$ and $b_{i}$, respectively; the last generator $h$ represents the $S^1$ fiber.

\paragraph{The group of the link}

Consider an arrow diagram of a link $L$ in a Seifert fibered space $M$ with base surface $F$ and fundamental polygon $G$. Fix an orientation on the link $L$ and on its arrow diagram. 

The overpasses of the link may encounter the boundary of $G$: when this happens, we index this boundary points with the following rule. Fix one of the corners of $G$ as a base point $\ast$; we may assume that it is the left corner of the edge $a_{1}$ as in Figure \ref{fig:generdiag}. 
Starting from the base point and going counterclockwise on the boundary of the diagram, we index by $+1, \ldots, +t$ only the points on the edges oriented according to that direction. The edges with the opposite orientation have the points of the links labeled by $-1, \ldots, -t$, so that $+ i$ and $-i$ are identified.

\begin{figure}[htb]
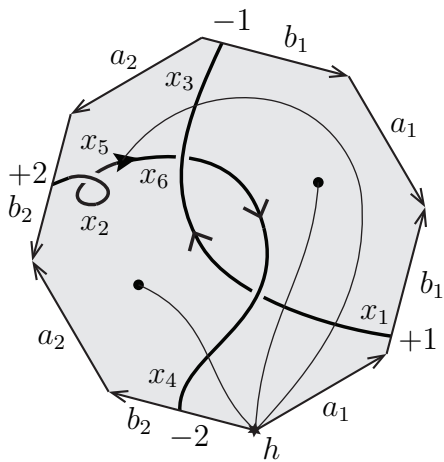

\centering
\begin{overpic}[page=42]{seif}
\put(26,4){$b_2$}\put(5,24){$a_2$}\put(-2,54){$b_2$}\put(22,90){$a_2$}   
\put(63,94){$b_1$}\put(88,74){$a_1$}\put(95,36){$b_1$}\put(72,7){$a_1$} 
\put(36,0){$-2$} \put(-2,62){$+2$}\put(46,98){$-1$}\put(90,23){$+1$} 
\put(31,15){$x_4$}\put(35,85){$x_3$}
\put(15,51){$x_2$}\put(81,30){$x_1$}
\put(15,70){$x_5$}\put(29,62){$x_6$}
\put(58,-3){$h$}
\end{overpic}
\caption{Reading the group generators on the arrow diagram.}
\label{fig:generdiag}
\end{figure}

Each time an arrow occurs, we use the convention that a new overpass begins. 
Moreover, we should assume that no overpass both starts and ends on the boundary or on an arrow. If that were the case, perform an $\Omega_{1}$ move on the overpass.

For the presentation of the group of the link, we use the generators $h$ and  $\{a_i , b_i \}_{i \in \{ 1, \ldots, g \} }$ and the indexes $\gamma_{i}$ and $\delta_{i}$ described for the fundamental group of $M$ in the paragraph above, moreover, we add the generators and relations described below.

We associate to each overpass a loop $x_{i}$, that is oriented by the left hand rule, according to the orientation of the overpass. The indexation of the generators associated to the overpasses should respect the following rule: $x_{1}, \ldots, x_{t}$ are the generators of those overpasses that end on the boundary points $+1, \ldots, +t$,  the generators $x_{t+1}, \ldots, x_{2t}$ correspond to $-1, \ldots, -t$, the generators $x_{2t+1}, \ldots, x_{2t+n}$ correspond to the overpasses before the arrows (considering the arrow orientation), the generators $x_{2t+n+1}, \ldots, x_{2t+2n}$ correspond to the overpasses after the arrows and finally $x_{2t+2n+1},\ldots, x_{r}$ are the remaining ones. Refer to Figure \ref{fig:gener} for an example. For the loops  $x_{i}$ with $i=1, \ldots, 2t$ and $i=2t+1, \ldots, 2t+2n$ we should add a sign $\epsilon_{i}=+1$ if the overpass, according to the orientation, enters from the boundary, $-1$ otherwise. Clearly $\epsilon_{i}=-\epsilon_{t+i},\  \forall i=1, \ldots, t $ and $\epsilon_{2t+i}=-\epsilon_{2t+n+i} \ \forall i=2t+1, \ldots, 2t+n $.

\begin{figure}[htb]
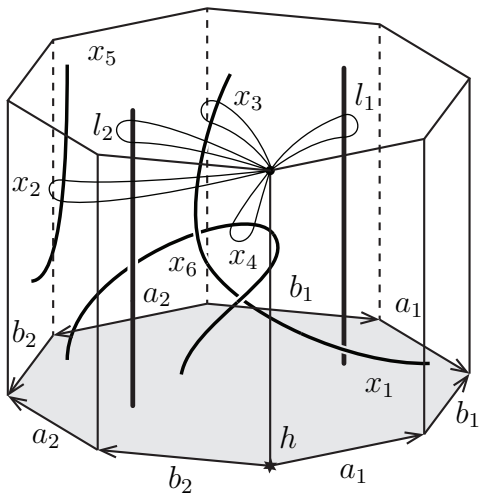

\centering
\begin{overpic}[page=43]{seif}
\put(36,0){$b_2$}\put(9,9){$a_2$}\put(5,29){$b_2$}\put(31,37){$a_2$}   
\put(60,38){$b_1$}\put(81,35){$a_1$}\put(93,13){$b_1$}\put(70,2){$a_1$} 
\put(75,19){$x_1$}\put(5,59){$x_2$}\put(49,76){$x_3$} 
\put(48,44.5){$x_4$}\put(20,85){$x_5$}\put(36,43){$x_6$}
\put(73,76){$l_1$}\put(21,70){$l_2$}
\put(58,8){$h$}  
\end{overpic}
\caption{The loops representing the group generators.}
\label{fig:gener}
\end{figure}

For each exceptional fiber $(\alpha_{i},\beta_{i})$, $i=1, \ldots, k$, let $l_{i}$ denote the generator associated to the fiber  and fix a path on the diagram, connecting $\ast$ to the fiber point.
Also for each arrow $v_{1}, \ldots, v_{n}$ we fix a path on the diagram connecting $\ast$ with the arrow.
These paths can not intersect each other and must follow the index order (before the fiber ones and then the arrow ones), as shown in Figure \ref{fig:generdiag}; we may also require that they are in general position with respect to the link projection.
If we read the overpass generators that we meet along the paths we compose the word $y_{i}$ for the fibers and the word $z_{j}$ for the arrows. In the case of Figure~\ref{fig:generdiag} we have $y_1=x_4$, $y_2=x_1$, and $z_1=x_1x_3^{-1}$.

From now on we assume $F$ is orientable. For the non-orientable case, see Theorem \ref{NorGroup}.

It is useful to label also the relations of the group presentation.
As usual, $W_{1}, \ldots, W_{s}$ denote the Wirtinger relations, that is a relation for every crossing as shown in Figure \ref{fig:wirt}.

\begin{figure}[htb]
\centering
\subfigure[$x_i x_k x_j^{-1} x_k^{-1} = 1$]{\begin{overpic}[page=7]{seif}
\put(52,6){$x_i$}\put(54,35){$x_j$}\put(30,29){$x_k$}
\end{overpic}\label{fig:wirt1}}\hspace{.5cm}
\subfigure[$x_i x_k^{-1} x_j^{-1} x_k = 1$]{\begin{overpic}[page=8]{seif}
\put(52,6){$x_k$}\put(55,32){$x_j$}\put(30,31){$x_i$}
\end{overpic}\label{fig:wirt2}}
\caption{Wirtinger relations.}
\label{fig:wirt}
\end{figure}

We introduce the inner automorphism (conjugation) of a group $G$ in order to simplify the group relations:
$$ \begin{array}{ccccc} \forall g \in G, & C(g): & G & \to & G, \\  &  & x & \mapsto & gxg^{-1}.\end{array}$$

For example, the Wirtinger relations may be rewritten as $x_{i}=C(x_{k})(x_{j})$ and $x_{i}=C(x_{k}^{-1})(x_{j})$, respectively.

The surface relation is
$$ F: \ \big( \prod_{i=1}^{g} [a_{i}, b_{i} ] \big) \big( \prod_{i=1}^{k} l_{i}^{-\beta_{i}} \big) \big( \prod_{i=2t+1}^{2t+n} x_{i}^{-\epsilon_{i}} \big)   =1. $$
For every $j=1, \ldots, g$ there are two relations $A_{j}$ and $B_{j}$, associated to the edges $a_{j}$ and $b_{j}$ of $F$:
\begin{align*}
A_{j}: &\  C \bigg( \big( \prod_{i=1}^{j-1} [a_{i},b_{i}] \big)^{-1} \bigg) \bigg(\prod_{ +i \textrm{ on } a_{j}} x_{i}^{\epsilon_{i}} \bigg) \cdot (a_{j} h a_{j}^{-1} h^{-\gamma_{j}} ) =1 \\
B_{j}: &\   C \bigg( a_{j}^{-1}   \big( \prod_{i=1}^{j-1} [a_{i},b_{i}] \big)^{-1} \bigg) \bigg( \prod_{ +i \textrm{ on } b_{j}} x_{i}^{\epsilon_{i}} \bigg) \cdot (b_{j} h b_{j}^{-1} h^{-\delta_{j}} )  =1 
\end{align*}
For every $j=1, 	\ldots, k$ there is a relation associated to the fiber $(\alpha_{j}, \beta_{j})$ 
$$CF_{j}: \  l_{j}^{-\beta_{j}}= C(h^{-1} y_{j})( l_{j}^{-\beta_{j}}).$$
For every $j=1, 	\ldots, n$, the relation associated to the arrow $v_{j}$ is 
$$CV_{j}: \ x_{2t+n+j}^{\epsilon_{2t+n+j}}= C(h^{-1} z_{j})( x_{2t+j}^{-\epsilon_{2t+j}} ). $$
The relation $CX_{j}$ for every $j=1, 	\ldots, t$ may have four different forms according to the following cases.  
The endpoint $+j$ belongs to the edge $a_{i}$ and $\gamma_{i}=+1$
$$CX_{j}: \ x_{t+j}^{\epsilon_{t+j}}= C \bigg( \big(\prod_{\iota=1}^{i} [a_{\iota},b_{\iota}] \big) b_{i} \big(\prod_{\iota=1}^{i-1} [a_{\iota},b_{\iota}] \big)^{-1} \bigg) \bigg( x_{j}^{-\epsilon_{j}} \bigg)  .$$
The endpoint $+j$ belongs to the edge $b_{i}$ and $\delta_{i}=+1$
$$ CX_{j}: \ x_{t+j}^{\epsilon_{t+j}}= C \bigg( \big(\prod_{\iota=1}^{i} [a_{\iota},b_{\iota}] \big)  a_{i}^{-1}  \big(\prod_{\iota=1}^{i-1} [a_{\iota},b_{\iota}] \big)^{-1} \bigg) \bigg( x_{j}^{-\epsilon_{j}} \bigg) .$$
The endpoint $+j$ belongs to the edge $a_{i}$ and $\gamma_{i}=-1$
\begin{multline*} 
CX_{j}: \ x_{t+j}^{\epsilon_{t+j}}= C \bigg( \big(\prod_{\iota=1}^{i} [a_{\iota},b_{\iota}] \big) b_{i} h  \big(\prod_{\iota=1}^{i-1} [a_{\iota},b_{\iota}] \big)^{-1} \big(\prod_{\substack{\iota<j \\ +\iota \textrm{ on } a_i}} x_{\iota}^{\epsilon_{\iota}} \big) \bigg) \bigg( x_{j}^{\epsilon_{j}} \bigg). 
 \end{multline*}
The endpoint $+j$ belongs to the edge $b_{i}$ and $\delta_{i}=-1$
$$CX_{j}: \ x_{t+j}^{\epsilon_{t+j}}= C \bigg( \big(\prod_{\iota=1}^{i} [a_{\iota},b_{\iota}] \big) h  a_{i}^{-1}  \big(\prod_{\iota=1}^{i-1} [a_{\iota},b_{\iota}] \big)^{-1}   \big(\prod_{\substack{\iota<j \\ +\iota \textrm{ on } b_i}} x_{\iota}^{\epsilon_{\iota}} \big) \bigg) \bigg( x_{j}^{\epsilon_{j}} \bigg).$$
Finally, for every $i=1, \ldots, k$, the relation $L_{j}$ of surgery is
 $$L_{j}: \ l_{j}^{\alpha_{j}} = y_{j}^{-1}h. $$

\begin{theorem}\label{OrGroup}
Given a link $L$ in a Seifert fibered space $M$ (assume that the base surface $F$ is orientable) with an arrow diagram satisfying the above condition, we get the following presentation for the group of the link:
\begin{align*}
\pi_{1}(M \smallsetminus  L, \ast) = \; & \langle x_{1}, \ldots , x_{r}, a_{1}, b_{1}, \ldots, a_{g}, b_{g}, h, l_{1}, \ldots, l_{k} \ | \\
& | \ W_{1}, \ldots ,W_{s}, F, A_{1}, B_{1},  \ldots, A_{g}, B_{g}, \\
& CF_{1}, \ldots, CF_{k}, CV_{1}, \ldots, CV_{n}, CX_{1}, \ldots, CX_{t},    L_{1}, \ldots, L_{k}  \rangle.
\end{align*}

\end{theorem}


\begin{remark}
When the genus of $F$ is zero, by applying a finite sequence of $\Omega_{9}^{O}$ moves, we can assume that the link has no boundary points. This produces a major simplification of the group presentation.
\end{remark}


In the case $F$ is non-orientable, consider the same generators except for the $b_{i}$-s that are missing in the boundary identification of $F$. The relations are modified as follows.
 
As usual, $W_{1}, \ldots, W_{s}$ denote the Wirtinger relations.
The surface relation becomes
$$ F:  \ \big( \prod_{i=1}^{g} a_{i}^{2} \big) \big( \prod_{i=1}^{k} l_{i}^{-\beta_{i}} \big) \big( \prod_{i=2t+1}^{2t+n} x_{i}^{-\epsilon_{i}} \big)   =1. $$
For every $j=1, \ldots, g$, the edge relation is
$$A_{j}: \  C \bigg( \big( \prod_{i=1}^{j-1} a_{i}^{2} \big)^{-1} \bigg) \bigg(\prod_{ +i \textrm{ on } a_{j}} x_{i}^{\epsilon_{i}} \bigg) \cdot (a_{j} h a_{j}^{-1} h^{-\gamma_{j}} ) =1 
   $$
For every $j=1, 	\ldots, k$, there is a fiber relation
$$CF_{j}: \  l_{j}^{-\beta_{j}}= C(h^{-1} y_{j})( l_{j}^{-\beta_{j}}).$$
For every $j=1, 	\ldots, n$, the relation associated to the arrow is 
$$CV_{j}: \ x_{2t+n+j}^{\epsilon_{2t+n+j}}= C(h^{-1} z_{j})( x_{2t+j}^{-\epsilon_{2t+j}} ). $$
The relation $CX_{j}$ for every $j=1, 	\ldots, t$ depends on the $i \in \{1, \ldots, g \}$ such that the endpoint $+j$ belongs to the edge $a_{i}$; the relation may have two different forms, according to $\gamma_{i}=\pm 1$.  
 If $\gamma_{i}=+1$ the relation is
$$CX_{j}: \ x_{t+j}^{\epsilon_{t+j}}= C \bigg( \big(\prod_{\iota=1}^{i-1} a_{\iota}^{2} \big) a_{i} \big(\prod_{\iota=1}^{i-1} a_{\iota}^{2} \big) \bigg) \bigg( x_{j}^{\epsilon_{j}} \bigg)  .$$
Otherwise
\begin{multline*} 
CX_{j}: \ x_{t+j}^{\epsilon_{t+j}}= C \bigg( \big(\prod_{\iota=1}^{i-1} a_{\iota}^{2} \big) a_{i} h  \big(\prod_{\iota=1}^{i-1} a_{\iota}^{2} \big)^{-1} \big(\prod_{\substack{\iota<j \\ +\iota \textrm{ on } a_i}} x_{\iota}^{\epsilon_{\iota}} \big) \bigg) \bigg( x_{j}^{-\epsilon_{j}} \bigg). 
 \end{multline*}
Finally, for every $i=1, \ldots, k$, the relation $L_{j}$ of surgery is
 $$L_{j}: \ l_{j}^{\alpha_{j}} = y_{j}^{-1}h. $$

\begin{theorem}\label{NorGroup}
Given a link $L$ in a Seifert fibered space $M$ (assume that the base surface $F$ is non-orientable) with an arrow diagram satisfying the above condition, we get the following presentation for the group of the link:
\begin{align*}
\pi_{1}( &M \smallsetminus  L, \ast)= \langle x_{1}, \ldots , x_{r}, l_{1}, \ldots, l_{k}, a_{1},  \ldots, a_{g}, h \ | \ W_{1}, \ldots ,W_{s}, F,\\
 &A_{1}, \ldots, A_{g}, CF_{1}, \ldots, CF_{k}, CV_{1}, \ldots, CV_{n}, CX_{1}, \ldots, CX_{t},    L_{1}, \ldots, L_{k}  \rangle.
\end{align*}

\end{theorem}

We will prove only the orientable case, the proof of the unorientable case can be made in the same fashion.

\begin{proof}[Proof of Theorem \ref{OrGroup}]

Let $G$ be the fundamental polygon of the surface $F$, base of the Seifert fibered space $M$. Let $D_{1}, \ldots, D_{k}$ be $k$ disjoint disks on $\textrm{int}\,G$, corresponding to the surgeries. Denote by $G_{0}$ the space $G \smallsetminus \{ D_{1}, \ldots, D_{k} \}$.
Consider the space $(G_{0} \times [0,1])/_\sim$, where $\sim$ is the equivalence relation that identifies the points of $G_{0} \times \{ 0 \}$ to the corresponding points of $G_{0} \times \{ 1 \}$, moreover the relation $\sim$ identifies the points of the lateral surface $\partial G \times [0,1]$, according to the boundary labels $a_{i}$ or $b_{i}$ and the signs $\gamma_{i}$ or $\delta_{i}$.
Let $P : (G_{0} \times [0,1]) \to (G_{0} \times [0,1])/_\sim$ be the quotient map.
The Seifert fibered space $M$ is the result of the suitable $(\alpha_{i},\beta_{i})$-fillings on $P(G_{0} \times [0,1]) $.

Consider the link $L \subset M$. Up to small isotopies we can assume that $L$ is all contained inside $P(G_{0} \times [0,1])$, and can be represented as a system of arcs inside $G_{0} \times [0,1]$, that may end on the boundary. 

The first goal is to compute $\pi_{1}(P(G_{0} \times [0,1])  \smallsetminus L, \ast)$, where $\ast$ is the basepoint on $G$ fixed as in Figure \ref{fig:gener}. In order to get a presentation of this group by the Seifert-Van Kampen theorem, we split $P(G_{0} \times [0,1])  \smallsetminus L$ into two parts.  The first part is the tabular neighbourhood $N(P(\partial G_{0} \times [0,1])  \smallsetminus L)$ of $P(\partial G_{0} \times [0,1])  \smallsetminus L)$ and the second part is the ``internal'' part $\textrm{int}(G_{0} \times [0,1])  \smallsetminus L$. Note that the first part deformation retracts to $P(\partial G_{0} \times [0,1])  \smallsetminus L)$. The intersection between the two parts deformation retracts to $(\partial G_{0} \times [0,1])  \smallsetminus L$.

As in the Wirtinger theorem for knots in $S^{3}$, the fundamental group of $\textrm{int}(G_{0} \times [0,1])  \smallsetminus L$ can be presented
by the generators $x_{1}, \ldots, x_{r}$ associated to the overpasses, by the generators $q_{1}, \ldots, q_{k}$ corresponding to the holes for surgeries and by the Wirtinger relations:
$$ \pi_{1}(\textrm{int}(G_{0} \times [0,1])  \smallsetminus L, \ast)=\langle x_{1}, \ldots, x_{r}, q_{1}, \ldots, q_{k}  \ | \  W_{1}, \ldots, W_{s} \ \rangle. $$

The space  $P(\partial G_{0} \times [0,1])  \smallsetminus L$ can be described by the following CW-complex:
\begin{description}
\item[0-complex] the base point $\ast$;
\item[1-complexes] the loops $a_{1}, b_{1}, \ldots, a_{g}, b_{g}$ of the surface $F$, the loop $h$ of the fibration $S^{1}$, the loops $\widetilde{q_{1}}, \ldots, \widetilde{q_{k}}$ for the surgery holes, the loops $v_{1}, \ldots, v_{n}$ corresponding to the arrows on the diagram, that means we have $n$ holes in $P(\partial G_{0} \times \{0,1\})$ created by $L$, the loops $d_{1}, \ldots, d_{t}$ on the lateral surface $P(\partial G \times [0,1])$ corresponding to the holes created by $L$;
\item[2-complexes] there is one 2-complex that represents the surface \hbox{$(G_{0} \times \{0,1\}) \smallsetminus L $} and other $2g$ 2-complexes corresponding to the surfaces $A_{1}, B_{1}, \ldots, A_{g}, B_{g}$ that are parts of the lateral surface $P(\partial G \times [0,1]) \smallsetminus L$.
\end{description}
As a consequence, since the maximal tree is trivial, each $1$-complex is a generator and each $2$-complex is a relation. 
\begin{align*}
 \pi_{1}&(P(\partial G_{0} \times [0,1])  \smallsetminus L, \ast) =  \langle \ a_{1}, b_{1}, \ldots, a_{g}, b_{g}, h,\widetilde{q_{1}}, \ldots, \widetilde{q_{k}}, \\
  & v_{1}, \ldots, v_{n}, d_{1}, \ldots, d_{t} \ | \ \big(\prod_{i=1}^{g}[a_{i}, b_{i}]\big) \big(\prod_{i=1}^{k}\widetilde{q_{i}}\big) \big(\prod_{i=1}^{n}v_{i}\big) =1,  \\
& \big(\prod_{i \ | \ d_{i} \in A_{j}} d_{i} \big) a_{j}h a_{j}^{-1} h^{-\gamma_{j}}=1, \big(\prod_{i \ | \ d_{i} \in B_{j}} d_{i} \big) b_{j}h b_{j}^{-1} h^{-\delta_{j}}=1, \ \forall j \in \{ 1, \ldots, g \} \, \rangle 
\end{align*}

The intersecting surface $(\partial G_{0} \times [0,1])  \smallsetminus L$ is a sphere with holes, hence its fundamental group is free and we label differently the generators according to the hole type: $q_{1}, \ldots, ,q_{k}$ (resp. $q_{k+1}, \ldots, ,q_{2k}$) are the surgery holes in $G \times \{ 1 \}$ (resp. $G \times \{ 0 \}$), $v_{1}, \ldots, v_{n}$  (resp. $v_{n+1}, \ldots, v_{2n}$) correspond to the arrow holes in  $G \times \{ 1 \}$ (resp. $G \times \{ 0 \}$) and $d_{1}, \ldots, d_{2t}$ correspond to the lateral surface holes, indexed according to the corresponding overpasses indexation. As a consequence, $ \pi_{1}((\partial G_{0} \times [0,1])  \smallsetminus L, \ast)$ is generated by $  q_{1}, \ldots, q_{2k}, v_{1},\ldots, v_{2n}, d_{1}, \ldots, d_{2t}$. 

By applying the Seifert-Van Kampen theorem we get the presentation: 
\begin{align*}
 \pi_{1}&(P(G_{0} \times [0,1])  \smallsetminus L, \ast) =\langle \ x_{1}, \ldots, x_{r}, q_{1}, \ldots, q_{k},  a_{1}, b_{1}, \ldots, a_{g}, b_{g}, h, \\
& \widetilde{q_{1}}, \ldots, \widetilde{q_{k}}, v_{1}, \ldots, v_{n}, d_{1}, \ldots, d_{t} \ | \ W_{1}, \ldots, W_{s}, \\
  & \big(\prod_{i=1}^{g}[a_{i}, b_{i}]\big) \big(\prod_{i=1}^{k}\widetilde{q_{i}}\big) \big(\prod_{i=1}^{n}v_{i}\big) =1,  \\
   &\big(\prod_{i \ | \ d_{i} \in A_{j}} d_{i} \big) a_{j}h a_{j}^{-1} h^{-\gamma_{j}}=1, \\
   & \big(\prod_{i \ | \ d_{i} \in B_{j}} d_{i} \big) b_{j}h b_{j}^{-1} h^{-\delta_{j}}=1, \ \forall j \in \{ 1, \ldots, g \} , \\
    & q_{j}=\widetilde{q_{j}}, q_{j}= C(h^{-1} y_{j}) (\widetilde{q_{j}}), \forall j \in \{ 1, \ldots, k \} , \\
      & v_{j}=x_{2t+j}^{-\epsilon_{2t+j}}, v_{j}= C(h^{-1} z_{j}) (x_{2t+n+j}^{\epsilon_{2t+n+j}} ) , \forall j \in \{ 1, \ldots, n \} , \\
  & x_{j}^{\epsilon_{j}}= C \bigg( \big(\prod_{\iota=1}^{i-1} [a_{\iota},b_{\iota}] \big) \bigg) \bigg(  d_{j}\bigg)  ,  \forall j \in \{ 1, \ldots, t \} \textrm{ s.t. } \exists i \textrm{ with } x_{j} \in A_{i} , \\
   &x_{j}^{\epsilon_{j}}= C \bigg( \big(\prod_{\iota=1}^{i-1} [a_{\iota},b_{\iota}] \big) a_{i} \bigg) \bigg(  d_{j}\bigg) , \forall j \in \{ 1, \ldots, t \} \textrm{ s.t. } \exists i \textrm{ with }  x_{j} \in B_{i} , \\
  & x_{t+j}^{\epsilon_{t+j}}= C \bigg( \big(\prod_{\iota=1}^{i} [a_{\iota},b_{\iota}] \big) b_{i}  \bigg) \bigg(d_{j}^{-1} \bigg),  \forall j \in \{ 1, \ldots, t \} \textrm{ s.t. } \exists i \textrm{ with }  x_{j} \in A_{i} \textrm{ and } \gamma_{i}=1 ,\\
   &x_{t+j}^{\epsilon_{t+j}}= C \bigg( \big(\prod_{\iota=1}^{i} [a_{\iota},b_{\iota}] \big) \bigg) \bigg(d_{j}^{-1} \bigg) ,  \forall j \in \{ 1, \ldots, t \} \textrm{ s.t. } \exists i \textrm{ with }  x_{j} \in B_{i} \textrm{ and } \delta_{i}=1, \\
  & x_{t+j}^{\epsilon_{t+j}}= C \bigg( \big(\prod_{\iota=1}^{i} [a_{\iota},b_{\iota}] \big) b_{i} h \big(\prod_{\iota < j | d_{\iota} \in A_i} d_{\iota} \big) \bigg) \bigg( d_{j} \bigg), \\
 &\forall j \in \{ 1, \ldots, t \} \textrm{ s.t. } \exists i \textrm{ with }  x_{j} \in A_{i} \textrm{ and } \gamma_{i}=-1 ,\\
   &x_{t+j}^{\epsilon_{t+j}}= C \bigg( \big(\prod_{\iota=1}^{i} [a_{\iota},b_{\iota}] \big) h \big(\prod_{\iota<j | d_{\iota} \in B_i} d_{\iota} \big) \bigg) \bigg( d_{j} \bigg)  , \\
 &\forall j \in \{ 1, \ldots, t \} \textrm{ s.t. } \exists i \textrm{ with }  x_{j} \in B_{i} \textrm{ and } \delta_{i}=-1  \ \rangle. 
 \end{align*}

After the deletion of the generators $ \widetilde{q_{1}}, \ldots, \widetilde{q_{k}} , v_{1}, \ldots, v_{n}, d_{1}, \ldots, d_{t}$, the presentation becomes:
\begin{align*}
 \pi_{1}&(P(G_{0} \times [0,1])  \smallsetminus L, \ast) =\langle \ x_{1}, \ldots, x_{r}, q_{1}, \ldots, q_{k}, a_{1}, b_{1}, \ldots, a_{g}, b_{g}, h \ | \\
 &  | \ W_{1}, \ldots, W_{s}, \\
  & \big(\prod_{i=1}^{g}[a_{i}, b_{i}]\big) \big(\prod_{i=1}^{k}q_{i}\big) \big(\prod_{i=2t+1}^{2t+n}x_{i}^{-\epsilon_{i}}\big) =1, \\ 
   &C\bigg(  \big(\prod_{\iota=1}^{j-1} [a_{\iota},b_{\iota}] \big)^{-1} \bigg) \bigg(\prod_{i \ | \ x_{i} \in A_{j}} x_{i}^{\epsilon_{i}} \bigg)  a_{j}h a_{j}^{-1} h^{-\gamma_{j}}=1,\ \forall j \in \{ 1, \ldots, g \} ,\\
   &C \bigg(  a_{j}^{-1}  \big(\prod_{\iota=1}^{j-1} [a_{\iota},b_{\iota}] \big)^{-1} \bigg) \bigg(\prod_{i \ | \ x_{i} \in B_{j}} x_{i}^{\epsilon_{i}} \bigg) a_{j}   b_{j}h b_{j}^{-1} h^{-\delta_{j}}=1, \ \forall j \in \{ 1, \ldots, g \} , \\
    &q_{j}= C(h^{-1} y_{j}) (q_{j}) , \forall j \in \{ 1, \ldots, k \} , \\
     & x_{2t+j}^{-\epsilon_{2t+j}}= C(h^{-1} z_{j}) (x_{2t+n+j}^{\epsilon_{2t+n+j}}), \forall j \in \{ 1, \ldots, n \} , \\
 &x_{t+j}^{\epsilon_{t+j}}= C \bigg( \big(\prod_{\iota=1}^{i} [a_{\iota},b_{\iota}] \big) b_{i} \big(\prod_{\iota=1}^{i-1} [a_{\iota},b_{\iota}] \big)^{-1} \bigg) \bigg( x_{j}^{-\epsilon_{j}} \bigg)  , \\
 &\forall j \in \{ 1, \ldots, t \} \textrm{ s.t. }\exists i \textrm{ with } x_{j} \in A_{i} \textrm{ and } \gamma_{i}=1 ,\\
&x_{t+j}^{\epsilon_{t+j}}= C \bigg( \big(\prod_{\iota=1}^{i} [a_{\iota},b_{\iota}] \big)  a_{i}^{-1}  \big(\prod_{\iota=1}^{i-1} [a_{\iota},b_{\iota}] \big)^{-1} \bigg) \bigg( x_{j}^{-\epsilon_{j}} \bigg)  , \\
 &\forall j \in \{ 1, \ldots, t \} \textrm{ s.t. }\exists i \textrm{ with } x_{j} \in B_{i} \textrm{ and } \delta_{i}=1, \\
  &x_{t+j}^{\epsilon_{t+j}}= C \bigg( \big(\prod_{\iota=1}^{i} [a_{\iota},b_{\iota}] \big) b_{i} h  \big(\prod_{\iota=1}^{i-1} [a_{\iota},b_{\iota}] \big)^{-1} \big(\prod_{\substack{\iota<j \\ +\iota \textrm{ on } a_i}} x_{\iota}^{\epsilon_{\iota}} \big) \bigg) \bigg( x_{j}^{\epsilon_{j}} \bigg) , \\
 &\forall j \in \{ 1, \ldots, t \} \textrm{ s.t. }\exists i \textrm{ with } x_{j} \in A_{i} \textrm{ and } \gamma_{i}=-1 ,\\
&x_{t+j}^{\epsilon_{t+j}}= C \bigg( \big(\prod_{\iota=1}^{i} [a_{\iota},b_{\iota}] \big) h  a_{i}^{-1}  \big(\prod_{\iota=1}^{i-1} [a_{\iota},b_{\iota}] \big)^{-1}   \big(\prod_{\substack{\iota<j \\ +\iota \textrm{ on } b_i}} x_{\iota}^{\epsilon_{\iota}} \big) \bigg) \bigg( x_{j}^{\epsilon_{j}} \bigg) , \\
 &\forall j \in \{ 1, \ldots, t \} \textrm{ s.t. }\exists i \textrm{ with } x_{j} \in B_{i} \textrm{ and } \delta_{i}=-1  \ \rangle. 
 \end{align*}

When we perform an $(\alpha_{1},\beta_{1})$-filling on $D_{1} \times [0,1]$ we can find the fundamental group of the result again through the Seifert-Van Kampen theorem.
The first space is the one above, the second space is a solid torus, whose fundamental group is presented by $  \langle \ l_{1} \ \rangle $, and the splitting surface is a torus, whose fundamental group is presented by $\langle \ l_{1}, m_{1}  \ \rangle $.
This operation can be done for each $k$ surgery, and at every step we apply the Seifert-Van Kampen theorem. The result is the previous group with the addition of the generators $l_{1}, \ldots, l_{k}$ and the relations $q_{j}=l_{j}^{-\beta_{j}}, l_{j}^{\alpha_{j}}=y_{j}^{-1} h ,  \forall j \in \{ 1, \ldots, k \}$.

After deleting the $q_{1}, \ldots, q_{k}$ generators with relations $ q_{j}=l_{j}^{-\beta_{j}} $, we get the desired result:
\begin{align*}
 \pi_{1}&(M  \smallsetminus L, \ast) =\langle \ x_{1}, \ldots, x_{r}, 
a_{1}, b_{1}, \ldots, a_{g}, b_{g}, h, l_{1}, \ldots, l_{k} \ | \\
 &| \ W_{1}, \ldots, W_{s}, 
 \big(\prod_{i=1}^{g}[a_{i}, b_{i}]\big) \big(\prod_{i=1}^{k}l_{i}^{-\beta_{i}}\big) \big(\prod_{i=2t+1}^{2t+n}x_{i}^{-\epsilon_{i}}\big) =1,  \\
  &C  \bigg( \big(\prod_{\iota=1}^{j-1} [a_{\iota},b_{\iota}] \big)^{-1} \bigg) \bigg(\prod_{i \ | \ x_{i} \in A_{j}} x_{i}^{\epsilon_{i}} \bigg) a_{j}h a_{j}^{-1} h^{-\gamma_{j}}=1,\ \forall j \in \{ 1, \ldots, g \} ,\\
    &C \bigg( a_{j}^{-1}  \big(\prod_{\iota=1}^{j-1} [a_{\iota},b_{\iota}] \big)^{-1}  \bigg) \bigg(\prod_{i \ | \ x_{i} \in B_{j}} x_{i}^{\epsilon_{i}} \bigg)   b_{j}h b_{j}^{-1} h^{-\delta_{j}}=1, \ \forall j \in \{ 1, \ldots, g \} , \\
 &l_{i}^{-\beta_{i}}= C(h^{-1} y_{j})( l_{i}^{-\beta_{i}} ) , \forall j \in \{ 1, \ldots, k \} , \\
 &x_{2t+n+j}^{\epsilon_{2t+n+j}}   = C(h^{-1} z_{j})(x_{2t+j}^{-\epsilon_{2t+j}}), \forall j \in \{ 1, \ldots, n \} , \\
 &x_{t+j}^{\epsilon_{t+j}}= C \bigg( \big(\prod_{\iota=1}^{i} [a_{\iota},b_{\iota}] \big) b_{i} \big(\prod_{\iota=1}^{i-1} [a_{\iota},b_{\iota}] \big)^{-1} \bigg) \bigg( x_{j}^{-\epsilon_{j}} \bigg)  , \\
 &\forall j \in \{ 1, \ldots, t \} \textrm{ s.t. }\exists i \textrm{ with } x_{j} \in A_{i} \textrm{ and } \gamma_{i}=1 ,\\
&x_{t+j}^{\epsilon_{t+j}}= C \bigg( \big(\prod_{\iota=1}^{i} [a_{\iota},b_{\iota}] \big)  a_{i}^{-1}  \big(\prod_{\iota=1}^{i-1} [a_{\iota},b_{\iota}] \big)^{-1} \bigg) \bigg( x_{j}^{-\epsilon_{j}} \bigg)  , \\
 &\forall j \in \{ 1, \ldots, t \} \textrm{ s.t. }\exists i \textrm{ with } x_{j} \in B_{i} \textrm{ and } \delta_{i}=1, \\
  &x_{t+j}^{\epsilon_{t+j}}= C \bigg( \big(\prod_{\iota=1}^{i} [a_{\iota},b_{\iota}] \big) b_{i} h  \big(\prod_{\iota=1}^{i-1} [a_{\iota},b_{\iota}] \big)^{-1} \big(\prod_{\substack{\iota<j \\ +\iota \textrm{ on } a_i}} x_{\iota}^{\epsilon_{\iota}} \big) \bigg) \bigg( x_{j}^{\epsilon_{j}} \bigg) , \\
 &\forall j \in \{ 1, \ldots, t \} \textrm{ s.t. }\exists i \textrm{ with } x_{j} \in A_{i} \textrm{ and } \gamma_{i}=-1 ,\\
&x_{t+j}^{\epsilon_{t+j}}= C \bigg( \big(\prod_{\iota=1}^{i} [a_{\iota},b_{\iota}] \big) h  a_{i}^{-1}  \big(\prod_{\iota=1}^{i-1} [a_{\iota},b_{\iota}] \big)^{-1}   \big(\prod_{\substack{\iota<j \\ +\iota \textrm{ on } b_i}} x_{\iota}^{\epsilon_{\iota}} \big) \bigg) \bigg( x_{j}^{\epsilon_{j}} \bigg) , \\
 &\forall j \in \{ 1, \ldots, t \} \textrm{ s.t. }\exists i \textrm{ with } x_{j} \in B_{i} \textrm{ and } \delta_{i}=-1 ,\\
&l_{j}^{\alpha_{j}}=y_{j}^{-1} h,  \forall j \in \{ 1, \ldots, k \} \ \rangle. 
\end{align*}

The description of the generators and the relations from the arrow diagram of the link is straightforward. 
\end{proof}

\begin{remark}
The group $ \pi_{1}((\partial G_{0} \times [0,1])  \smallsetminus L, \ast)$ is free with a generator less than the set used in the proof, hence it is possible do delete one of the relations between $CF_{i}$, $CV_{i}$ or $CX_{i}$.
\end{remark}


\section{The first homology group of links in Seifert Manifolds}\label{linkhomology}

\paragraph{Homology classes of knots in Seifert fibered space}
Let us recall a presentation of the first homology group of the Seifert fibered space itself, different from the ones that can be recovered from Equations \ref{gruppoFor} and \ref{gruppoFnor}. Namely, the following two presentations can be obtained from the fundamental groups of Theorem \ref{OrGroup} and \ref{NorGroup} by assuming $L= \emptyset$. If $F$ is orientable, 
\begin{align*}
H_{1}(M)= \langle \ & a_{1}, b_{1}, \ldots, a_{g}, b_{g}, h, l_{1}, \ldots, l_{k}  \  | \ l_{1}^{-\beta_{1}}\cdots l_{k}^{-\beta_{k}}=1, \\
&h^{1-\gamma_{j}}=1,  h^{1-\delta_{j}}=1, \forall j = 1, \ldots, g,   l_{i}^{\alpha_{i}}=h, \forall i = 1, \ldots, k \   \rangle,
\end{align*}
 whereas if $F$ is non-orientable, 
 \begin{align*}
 H_{1}(M)=\langle \ & a_{1},  \ldots, a_{g}, h, l_{1}, \ldots, l_{k} \ | \   a_{1}^{2}\cdots a_{g}^{2} l_{1}^{-\beta_{1}}\cdots l_{k}^{-\beta_{k}}=1,\\
 & h^{1-\gamma_{j}}=1, \forall j = 1, \ldots, g,  l_{i}^{\alpha_{i}}=h, \forall i = 1, \ldots, k \ \rangle.
 \end{align*}
Hence the homology group usually has a torsion part. The complete determination of the torsion part is quite complicated and is treated in \cite{BLPZ, BH}.

The homology class $[K] \in H_{1}(M)$ of a knot $K \subset M$ is an isotopy invariant. We can determine the homology class of a knot directly from its diagram. 

\begin{lemma}\label{homologyclass}
The coefficients $\eta_{a,1}, \ldots, \eta_{a,g}, \eta_{b,1}, \ldots, \eta_{b,g}, \eta_{h}, \eta_{l,1}, \ldots, \eta_{l,k}$ that determine the homology class $\alpha_{1}^{\eta_{a,1}} \cdots \alpha_{g}^{\eta_{a,1}} \beta_{1}^{\eta_{b,1}} \cdots   \beta_{g}^{\eta_{b,g}}  h^{\eta_{h}}  l_{1}^{\eta_{l,1}}  \cdots  l_{k}^{\eta_{l,k}} \in H_{1}(M)$ of the knot can be found by the following formulas:
$$ \eta_{a,j}=\sum_{\substack{k=1 \\ +k \textrm{ on } a_{j}}}^{t} \epsilon_{k},  
\quad \quad  
\eta_{b,j}=\sum_{\substack{k=1 \\ +k \textrm{ on } b_{j}}}^{t} \epsilon_{k},  
\quad \quad  
\eta_{h}=\sum_{k=1}^{n} \epsilon_{2t+k}, \quad \quad \forall j=1, \ldots, g.$$
Finally, by substituting all overpasses with a single generator, the word $y_{j}$ becomes a power of this generator; this power is $\eta_{l,j}$ for every $j=1, \ldots, k$; alternatively, $\eta_{l,j}$ is the winding number of the knot curve around the dot on the diagram representing the $(\alpha_{j},\beta_{j})$-surgery.
\end{lemma}

Be careful that two different indexing sets may represent the same homology class, because of the torsion. The proof is an easy generalization of \cite[Lemma 4]{CMM}. 
The case of links is similar, because we assign a homology class to each component with the same formulas regarding only the component we are considering.


\paragraph{The first homology group of the complement of a link}
Now let's turn our attention to the first homology group of the complement, $H_{1}(M \smallsetminus L)$.

Through the Hurewicz theorem we can compute the first homology group starting from the fundamental group; the abelianization of its presentation produces the following presentation of abelian groups, where $\nu$ is the number of components of $L = L_{1} \sqcup \ldots \sqcup L_{\nu}$.
Moreover, we denote the coefficients of the homology class of the component $L_{i}$ by $\eta_{a,j,i}$, $\eta_{b,j,i}$, $\eta_{h,i}$ and $\eta_{l,j,i}$, where the third index specifies the component.

\begin{theorem}\label{homologyMoFo}
If both $M$ and $F$ are orientable ($\gamma_{i}, \delta_{i}=+1 \ \forall i=1, \ldots, g $), the first homology group of a link $L$ in the Seifert fibered space $M$ is:
\begin{align}\label{homMorFor}
 H_{1}&(M  \smallsetminus L) =\langle \ g_{1}, \ldots, g_{\nu}, 
a_{1}, b_{1}, \ldots, a_{g}, b_{g}, h, l_{1}, \ldots, l_{k} \ | \nonumber \\
 & | \
 \big(\prod_{i=1}^{k}l_{i}^{-\beta_{i}}\big)= \big(\prod_{i=1}^{\nu}g_{i}^{\eta_{h,i}}\big) ,  
 l_{j}^{\alpha_{j}}=\big(\prod_{i=1}^{\nu} g_{i}^{-\eta_{l,j,i}} \big)  h,  \forall j \in \{ 1, \ldots, k \},\nonumber  \\
&\big(\prod_{i=1}^{\nu} g_{i}^{\eta_{a,j,i}} \big)=1, \big(\prod_{i=1}^{\nu} g_{i}^{\eta_{b,j,i}} \big) =1, \ \forall j \in \{ 1, \ldots, g \} \ \rangle. 
\end{align}

If $M$ is orientable but $F$ is non-orientable ($\gamma_{i}=-1, \forall i=1, \ldots, g$):
\begin{align}\label{homMorFnor}
 H_{1}&(M  \smallsetminus L) = \langle \ g_{1}, \ldots, g_{\nu}, 
a_{1}, \ldots, a_{g}, h, l_{1}, \ldots, l_{k} \ |\nonumber  \\
 &| \
 \big( \prod_{i=1}^{g} a_{i}^{2} \big) \big(\prod_{i=1}^{k}l_{i}^{-\beta_{i}}\big) = \big(\prod_{i=1}^{\nu}g_{i}^{\eta_{h,i}}\big) ,  
 l_{j}^{\alpha_{j}}=\big(\prod_{i=1}^{\nu} g_{i}^{-\eta_{l,j,i}} \big)  h,  \forall j \in \{ 1, \ldots, k \},\nonumber  \\
&\big(\prod_{i=1}^{\nu} g_{i}^{\eta_{a,j,i}} \big) h^{1-\gamma_{j}}=1, \ \forall j \in \{ 1, \ldots, g \} \ \rangle. 
\end{align}
In the same case, the presentation of the homology group with $\mathbb{Z}_{2}$ coefficients simplifies and becomes the same of Equation \ref{homMorFor}.
\end{theorem}

\begin{proof}
If both $F$  and $M$ are orientable, it is enough to abelianize the group presentation of Theorem \ref{OrGroup}, considering that the conjugation relations reduce to $x_{i}=x_{j}$ for every overpass corresponding to the same component. If $M$ is orientable but $F$ is not, then we abelianize the group presentation of Theorem \ref{NorGroup} and reduce all the $x_{i}$ to a generator for each component. 
\end{proof}

For the case where $M$ is non-orientable we cannot use the coefficients $\eta_{a,j,i}$, $\eta_{b,j,i}$, $\eta_{h,i}$ and $\eta_{l,j,i}$, instead we will denote by $\zeta_{a,j,i}$, $\zeta_{b,j,i}$, $\zeta_{h,i}$ and $\zeta_{l,j,i}$ the coefficients arising from the group abelianization.

\begin{theorem}\label{homologyMnFo}
If $M$ is non-orientable and $F$ is orientable:
\begin{align}\label{homMnorFor}
 H_{1}&(M  \smallsetminus L) = \langle \ g_{1}, \ldots, g_{\nu}, 
a_{1}, b_{1}, \ldots, a_{g}, b_{g}, h, l_{1}, \ldots, l_{k} \ | \nonumber\\
 &| \ \big(\prod_{i=1}^{k}l_{i}^{-\beta_{i}}\big) = \big(\prod_{i=1}^{\nu}g_{i}^{\zeta_{h,i}}\big) ,  
 l_{j}^{\alpha_{j}}=\big(\prod_{i=1}^{\nu} g_{i}^{-\zeta_{l,j,i}} \big)  h,  \forall j \in \{ 1, \ldots, k \}, \nonumber\\
&\big(\prod_{i=1}^{\nu} g_{i}^{\zeta_{a,j,i}} \big) h^{1-\gamma_{j}}=1, \big(\prod_{i=1}^{\nu} g_{i}^{\zeta_{b,j,i}} \big)  h^{1-\delta_{j}}=1, \ \forall j \in \{ 1, \ldots, g \} \ \rangle. 
\end{align}

If both $M$ and $F$ are non-orientable:
\begin{align}\label{homMnorFnor}
 H_{1}&(M  \smallsetminus L) = \langle \ g_{1}, \ldots, g_{\nu}, 
a_{1}, \ldots, a_{g}, h, l_{1}, \ldots, l_{k} \ | \nonumber\\
 &| \
 \big( \prod_{i=1}^{g} a_{i}^{2} \big) \big(\prod_{i=1}^{k}l_{i}^{-\beta_{i}}\big) = \big(\prod_{i=1}^{\nu}g_{i}^{\zeta_{h,i}}\big) ,  
 l_{j}^{\alpha_{j}}=\big(\prod_{i=1}^{\nu} g_{i}^{-\zeta_{l,j,i}} \big)  h,  \forall j \in \{ 1, \ldots, k \}, \nonumber\\
&\big(\prod_{i=1}^{\nu} g_{i}^{\zeta_{a,j,i}} \big) h^{1-\gamma_{j}}=1, \ \forall j \in \{ 1, \ldots, g \} \ \rangle. 
\end{align}

\end{theorem}

\begin{proof}
When $F$ is orientable, starting from the group presentations of Theorem \ref{OrGroup}, when we abelianize the conjugation relations, they reduce to $x_{i}=x_{j}$ for every overpass corresponding to the same component, except when $\gamma_{i}$ or $\delta_{i}$ are equal to $-1$; in this case the conjugation relations become $x_{i}=x_{j}^{-1}$ hence the substitution is $x_{i}=g_{k}^{-1}$ instead of $x_{i}=g_{k}$ as usual; the consequence is that the coefficients $\eta_{a,j,i}$, $\eta_{b,j,i}$, $\eta_{h,i}$ and $\eta_{l,j,i}$ coming from the homology class of $L_{k}$ are no more useful and we have to modify them into $\zeta_{a,j,i}$, $\zeta_{b,j,i}$, $\zeta_{h,i}$ and $\zeta_{l,j,i}$. If $F$ is non-orientable, when we abelianize the group of Theorem \ref{NorGroup}, the case $x_{i}=x_{j}^{-1}$ is produced when $\gamma_{i}$ or $\delta_{i}$ are equal to $1$.
\end{proof}


\paragraph{The rank of the first homology group}
If we consider the homology presentation of the complement of a link $L$ in the Seifert fibered space $M$, then (when $F$ is orientable) the generators $\alpha_{1}, \beta_{1}, \ldots, \alpha_{g}, \beta_{g}$ are always abelian free, while sometimes the generators $g_{1}, \ldots, g_{\nu}$ are eliminated or produce torsion. It is not possible to predict when this happens in an easy way.

A nice way to investigate this problem is the following long exact sequence, corresponding to the pair $(M, M\smallsetminus L)$: 
$$\cdots \rightarrow  H_{2}(M) \rightarrow H_{2}(M, M \smallsetminus L) \rightarrow H_{1}(M\smallsetminus L)  \rightarrow H_{1}(M) \rightarrow H_{1}(M, M \smallsetminus L) \rightarrow \cdots  $$

For all $i=0,\ldots, 3$, using excision and the fact that the homology of disjoint spaces is the direct sum of the respective homologies, it holds:
$$H_{i}( M, M \smallsetminus L) \cong H_{i}( \bigsqcup_{j=1}^{ \nu} T_{j} , \bigsqcup _{j=1}^{\nu} \partial T_{j} ) \cong \bigoplus_{j=1}^{\nu} H_{i}(T_{j}, \partial T_{j}). $$

From the Lefschetz duality we have $H_{i}(T, \partial T) \cong H^{3-i}(T)$, as a consequence $H_{2}( M, M \smallsetminus L) \cong \Z^{\nu}$ and $H_{1}( M, M \smallsetminus L) \cong 0$.

The long exact sequence becomes:
$$\cdots \rightarrow  H_{2}(M) \rightarrow \Z^{\nu} \rightarrow H_{1}(M\smallsetminus L)  \rightarrow H_{1}(M) \rightarrow 0  $$

Under the assumption that $M$ is a rational homology sphere ($\Q HS$), we have $H_{2}(M) \cong H^{1}(M) \cong \textrm{Hom}(H_{1}(M), \Z) \cong 0$ and the exact sequence simplifies to:
$$0 \rightarrow \Z^{\nu} \rightarrow H_{1}(M\smallsetminus L)  \rightarrow H_{1}(M) \rightarrow 0.$$

This sequence in general is not split, but some split cases may be identified directly from the homology presentation.

\begin{corollary}
If $M$ is a Seifert fibered space and $L$ is a local link (that is, contained inside a $3$-ball) then $H_1(M\smallsetminus K)\cong H_1(M) \oplus \mathbb{Z}^{\nu}$.
\end{corollary}
\begin{proof}
If the link is local, then looking to the diagram of the link we may assume that it has not got any boundary points, nor arrows, nor windings around the surgeries. As a consequence, the fundamental group presentation has only the relations of the Seifert fibered spaces and we see directly \mbox{$H_1(M\smallsetminus K)\cong H_1(M) \oplus \mathbb{Z}^{\nu}$.}
\end{proof}

It is not possible to assume just $L$ to be a null-homologous knot, as Example \ref{exK} shows.

\begin{example}\label{exK}
Let $M$ be a non-orientable Seifert fibered space with an orientable base surface and without surgeries. Assume $g(F)=1$ and $\gamma_{1}=1$, $\delta_{1}=-1$. Let $K$ be the null-homologous knot in $M$ depicted in Figure~\ref{fig:ex41}. Since $H_{1}(M)= \langle \ a,b,h \ | \ h^{2}=1 \ \rangle$, then $K$ is null-homologous even if it has two arrows. Equation \ref{homMnorFor} gives us the result $H_{1}(M \smallsetminus K)= \langle \ a,b,h,g \ | \ h^{2}=1, g^{2}=1 \ \rangle \cong \Z^{2} \oplus \Z_{2}^{2}$, that is to say, the sequence in this case is not split.
\begin{figure}[htb]
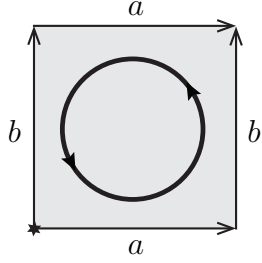

\centering
\begin{overpic}[page=44]{seif}
\put(48,0){$a$}\put(96,45){$b$}\put(48,96){$a$}\put(0,45){$b$}   
\end{overpic}
\caption{A knot diagram.}
\label{fig:ex41}
\end{figure}
\end{example}

\begin{corollary}
If $M$ is a circle bundle (that is, $F$ is orientable, $M$ is orientable and without surgeries) and $L$ is a null-homologous link (that is, all of its components are null-homologous) then $H_1(M\smallsetminus L)\cong H_1(M) \oplus \mathbb{Z}^{\nu}$.
\end{corollary}
\begin{proof}
Since the components of $L$ are null-homologous and all the generators of $H_{1}(M)$ are free, then the indexes $\mu_{a,i,j}$, $\mu_{b,i,j}$ and $\mu_{h,j}$ are all equal to zero: the presentations of Equation \ref{homMorFor} eliminates from the relations all the generators of the link, that is, the relations are exactly the ones of $H_{1}(M)$ and the exact sequence splits.
\end{proof}

A weaker problem asks if $rank(H_{1}(M \smallsetminus L))=rank(H_{1}(M))+\nu$. A lower bound for $rank(H_{1}(M \smallsetminus L))$ is guaranteed by the following condition.

\begin{corollary}\label{homsplit}
Let $M$ be an orientable Seifert fibered space with at least one surgery, with orientable base surface $F$ of genus $g$ and $L$ a $\nu$-components (all of them null-homologous) link in $M$ whose diagram has no boundary points (arrows are allowed). Then $rank(H_{1}(M \smallsetminus L)) \ge 2g +\nu$. 
\end{corollary}
\begin{proof}
The proof is done by simplifying the presentation of the first homology group. 
If the diagram has no boundary points, the homology presentation of Equation \ref{homMorFor} becomes:
\begin{align*}
 H_{1}(M  \smallsetminus L) =&\langle \ g_{1}, \ldots, g_{\nu}, 
a_{1}, b_{1}, \ldots, a_{g}, b_{g}, h, l_{1}, \ldots, l_{k} \ | \\
 &| \
 \big(\prod_{i=1}^{k}l_{i}^{-\beta_{i}}\big)= \big(\prod_{i=1}^{\nu}g_{i}^{\eta_{h,i}}\big) ,  
 l_{j}^{\alpha_{j}}=\big(\prod_{i=1}^{\nu} g_{i}^{-\eta_{l,j,i}} \big)  h,  \forall j \in \{ 1, \ldots, k \} \ \rangle. 
\end{align*}
 That is, we have $\nu + 2g+1+k$ generators and $1+k$ relations. As a consequence $rank(H_{1}(M \smallsetminus L))\ge 2g+\nu$.
 \end{proof}

Two easy conditions on $M$ so that the conditions of Corollary \ref{homsplit} are satisfied, are that $M$ has genus $0$ or that $M$ is a $\Q HS$.

Nevertheless, not only the exact sequence does not generally splits, but also the rank condition $rank(H_{1}(M \smallsetminus L))=rank(H_{1}(M))+\nu$ fails. Besides Example~\ref{exK}, another class of counterexamples is the following one.

\begin{example}
In case where $M$ is orientable and without surgeries, if the base surface is orientable, then $H_{1}(M)=\langle \ a_{1}, b_{1}, \ldots, a_{g}, b_{g}, h \ \rangle$. If the diagram of a link $L\subset M$ has no boundary points this means that the homology presentation of Equation \ref{homMorFor} becomes:
$$
 H_{1}(M  \smallsetminus L) =\langle \ g_{1}, \ldots, g_{\nu}, 
a_{1}, b_{1}, \ldots, a_{g}, b_{g}, h \ |  \big(\prod_{i=1}^{\nu}g_{i}^{\eta_{h,i}}\big) =1  \ \rangle. 
$$
 That is, we have $\nu + 2g+1$ generators, the same of $H_{1}(M)\oplus \Z^{\nu} $, but torsion may arise from the link generators. 
\end{example}

In conclusion, the torsion depends on too many factors in order to give an explicit formula for it. Once having the presentation, the torsion can be easily computed by the Euclidean algorithm (also through software like GAP or SAGE). This is very useful if we want to compute the twisted Alexander polynomial associated to the group presentation. In this context, the most important condition on $H_{1}(M \smallsetminus L)$ is that its rank is at least greater or equal to the number of components.

\begin{lemma}\label{rankone}
Given a link $L$ with $\nu$ components in a Seifert fibered space $M$, it holds that $rank(H_{1}(M \smallsetminus L)) \ge \nu$. 
\end{lemma}
\begin{proof}
If the base surface $F$ is orientable, by analyzing the presentations of the homology group of the link given in Equations \ref{homMorFor} and \ref{homMnorFor}, we have $\nu + 2g + k +1$ generators and $2g + k +1$ relations. Be careful that the $2g$  boundary relations involve only the link generators $g_{1}, \ldots, g_{\nu}$, so they eventually erase these  generators instead of the surface ones; moreover, each relation may erase at most one generator, so the free generators between $g_{1}, \ldots, g_{\nu}, a_{1}, b_{1}, \ldots, a_{g}, b_{g}$ are at least $2g+\nu-\textrm{min}(\nu, 2g) \ge \nu$. The other generators may cancel out with the other relations, so $rank(H_{1}(M \smallsetminus L)) =2g+\nu-\textrm{min}(\nu, 2g) \ge \nu$.

When $F$ is not orientable, a similar reasoning on Equations
\ref{homMorFnor} and \ref{homMnorFnor} brings to $rank(H_{1}(M \smallsetminus L))= g+\nu-\textrm{min}(\nu, g) \ge \nu$.
\end{proof}


\section{Twisted Alexander polynomial of links in Seifert Manifolds}\label{linkTAP}

The group of the link is a powerful invariant, but due to the word problem we cannot in general tell if two group presentations present different groups.
The Alexander polynomial associated to a group presentation often enables us to distinguish the groups. Twisted Alexander polynomials were introduced by Wada in \cite{Wa}. In this section we recall the description of the multi-variable invariants, the focus will be on a particular class of twisted polynomials that considers a $1$-dimensional representation of the group, in order to keep track of the torsion part of the link group.
The behaviour of the twisted polynomials on local links and under the connected sum is shown. 

\paragraph{Twisted Alexander polynomials for finitely presented groups}

We follow Turaev's construction of twisted Alexander polynomials \cite{T}. 

Let $\langle \ x_{1}, \ldots, x_{m} \ | \ r_{1}, \ldots, r_{n} \ \rangle$ be a finite presentation of a group $\pi$, and let $H$ be the abelianization of $\pi$. A relation $r_{i}$ is an element of the free group $F=F(x_{1}, \ldots, x_{m})$. 

The Fox derivative of $r_{i}$ with respect to $x_{j}$, denoted by $\partial r_{i} / \partial x_{j}$, is defined recursively by the following rules:
$$ \frac{\partial 1}{ \partial x_{j}} =0, \ \frac{\partial x_{i}}{ \partial x_{j}}=\delta_{i,j},  \ \frac{\partial x_{i}^{-1}}{ \partial x_{j}}=\delta_{i,j} x_{i}^{-1}  ,  \ \frac{\partial (ux_{i})}{ \partial x_{j}}= \frac{\partial u }{ \partial x_{j}} + u \frac{\partial x_{i} }{ \partial x_{j}},$$  
where $u$ is a word of $F(x_{1}, \ldots, x_{m})$. The result is an element of $\Z[F]$. 

Now consider the matrix $[\partial r_{i} / \partial x_{j}]_{i,j}$ and apply the projection $\Z[F] \to \Z[\pi] \to \Z[H]$ to its entries: the result is the Alexander-Fox matrix $A$. We may assume $n \ge m$, by adding trivial relations if necessary. For each integer $d$ such that $0 \le d < m$, the ideal $E_{d}(\pi) \subset \Z[H]$ is generated by the minors of $A$ of order $m-d$. Let $E_{d}(\pi)=\Z[H]$ if $d \ge m$. The ideals do not depend on the presentation of $\pi$. The ideal $E_{0}(\pi)$ is completely determined by $H$: it is $0$ if $H$ is infinite and it is generated by $\sum_{h \in H} h $ otherwise. The focus will be on $E_{1}(\pi)=E(\pi)$.

The twisted Alexander polynomials we consider will be numerated by $\sigma \in \textrm{Hom}(\textrm{Tors} H, \C^{\ast})$. Let $G=H/\textrm{Tors}H$ and fix a splitting $H=  \textrm{Tors}H  \times G$. Let $\tilde{\sigma} : \Z[H] \to \C[G]$ be the homomorphism that sends $fg$ with $f \in \textrm{Tors}H$ and $g \in G$ to $\sigma(f)g$. Since $\C[G]$ is an UFD, we can set $\Delta^{\sigma}(\pi)=\textrm{gcd } \tilde{\sigma}(E(	\pi))$. The $\textrm{gcd}$ is defined up to multiplication of elements of $G$ and of $\C^{\ast}$. 
The change of the choice of the splitting $H=  \textrm{Tors}H  \times G$ produces an element $\psi \in \textrm{Hom}(G, \textrm{Tors}H)$, hence the polynomial  $\Delta^{\sigma}(\pi)=\sum_{g \in G} d_{g}g$ where $d_{g} \in \C$, is transformed into $\sum_{g \in G} d_{g} \sigma(\phi(g)) g$. When $\sigma=1$, we get the classical Alexander polynomial $\Delta^{1}(\pi)=\C^{\ast}\Delta(\pi)$.

Observe that the ring $\C[G]$ is the ring $\C[z_{1}^{\pm1}, \ldots, z_{k}^{\pm1}]$, where $k= \textrm{rank} (H)$. In order to simplify the computations is useful to reduce from the multi-variable to the one-variable twisted Alexander polynomials by means of the projection $\C[z_{1}^{\pm1}, \ldots, z_{k}^{\pm1}] \to\C[z^{\pm1}]$

\paragraph{Twisted Alexander polynomials for links in Seifert fibered spaces}

Given a link $L$ in the Seifert fibered space $M$ we will denote by
$\Delta^{\sigma}_{L}=\Delta^{\sigma}(\pi_{1}(M \smallsetminus L, \ast))$ the twisted Alexander polynomials associated to it and numerated by $\sigma \in \textrm{Hom}(\textrm{Tors} H_{1}(M \smallsetminus L), \C^{\ast})$

Please note that if $H_{1}(M \smallsetminus L)$ is finite, the Alexander polynomial is trivial, hence not significant, for this reason we proved Lemma \ref{rankone}, that guarantees us that this case cannot occur.

\paragraph{Twisted Alexander polynomial for local links}

Recall that a link $L$ is local if it is contained inside a 3-ball $B^{3}$ embedded in $M$.
Let $\bar{L}$ denote the local link $L \subset B^{3}$ embedded into $\s3$. The following theorem holds.

\begin{theorem}\label{local}
Let $L \subset M$ be a  local link. Then 
\begin{itemize}
\item if $H_{1}(M)$ is infinite, $\Delta^{\sigma}_L=0$;
\item if $H_{1}(M)$ is finite, the classical Alexander polynomial gives  
$ \Delta^{\sigma}_L=|H| \cdot \Delta_{\bar{L}}$, while for each $\sigma \neq 1$, $\Delta^{\sigma}_L=0$.
\end{itemize}

\end{theorem}
\begin{proof}
We may assume that the link diagram has the link arcs all contained inside a disk with no punctures inside it and with no arrows.
The relations of the presentation of the fundamental group of $M \smallsetminus L$ are divided into two sets, one containing the Writinger relations that involves only the link generators, and the set of relations that involves only the manifold generators. 

Therefore the Alexander-Fox matrix $A$ splits into two blocks as follows:
$$A_L=\left(\begin{array}{cccccc}\ &A_{M}&\ &\  &0 &\ \\\ & 0 \ &\ & \ &A_{\bar{L}}&\ \end{array}\right),$$
where the matrices $A_{L}$, $A_{M}$ and $A_{\bar{L}}$ have respectively $m+n$, $m$ and $n$ columns.

As a consequence, if we denote by $\Delta_{d}=\textrm{gcd} \tilde{\sigma}(E_{d}(	\pi))$ and $\bar{L} \subset \s3$ the local link viewed as a link in the $3$-sphere, it holds $\Delta^{\sigma}_{L}=\Delta^{\sigma}_{1}= \textrm{gcd} (\Delta^{\sigma}_{1}(\pi_{1}(M))\cdot \Delta^{\sigma}_{0}(\pi_{1} ( \s3 \smallsetminus \bar{L})) , \Delta^{\sigma}_{0}(\pi_{1}(M))\cdot \Delta^{\sigma}_{1}(\pi_{1} (\s3 \smallsetminus \bar{L})) )$, because we are considering the $m+n-1$ minors, that are given by the combinations of $(m-1,n)$ and $(m, n-1)$ of minors of $A_{M}$ and $A_{\bar{L}}$.
Since $\Delta^{\sigma}_{n}(\pi_{1} (\s3 \smallsetminus \bar{L}))=0$ (because $H_{1}(\s3 \smallsetminus \bar{L})$ is infinite), we reduce to $\Delta^{\sigma}_{1}= \Delta^{\sigma}_{0}(\pi_{1}(M))\cdot \Delta^{\sigma}_{1}(\pi_{1} (\s3 \smallsetminus \bar{L})) )$.
Now if $H_{1}(M)$ if infinite, $ \Delta^{\sigma}_{0}(\pi_{1}(M))= \textrm{gcd} \tilde{\sigma}(E_{0}(	\pi_{1}(M)))=0$, hence $\Delta^{\sigma}_L=0$.
If $H_{1}(M)$ if finite and $\sigma=1$ then $\sum_{h \in H} h$ becomes $|H|$, so  $\Delta^{\sigma}_L=|H| \cdot \Delta_{\bar{L}}$.
If $H_{1}(M)$ if finite and $\sigma \neq 1$, then $H_{1}(M)=\Z_{p_{1}}\oplus\ldots\oplus\Z_{p_{k}}$ for some $k$; for each $\Z_{p_{i}}$, the corresponding generator is sent to a $p_{i}$-root of unity.  In this case $\sum_{h \in \Z_{p_{i}}} h=0$ and extending to $H$ we have $\sum_{h \in H} h=0$, so as before $ \Delta^{\sigma}_{0}(\pi_{1}(M))= \textrm{gcd} \tilde{\sigma}(E_{0}(	\pi_{1}(M)))=0$ and hence $\Delta^{\sigma}_L=0$.
\end{proof}

As a consequence a knot with a non trivial twisted (that is, $\sigma \neq 1$) Alexander polynomial cannot be local. See also Example \ref{es1}. Observe that in the case of lens spaces, where $H_{1}(L(p,q)) \cong \Z_{p}$, the second point of Theorem \ref{local} is exactly the result stated in \cite[Proposition 7]{CMM}.

\paragraph{Twisted Alexander polynomial for connected sum of links}

Let $L$ be a link in  $M$ such that it is a connected sum, that is to say $L=L_1\sharp L_2$ where $L_{1} \subset M$ and $L_{2} \subset \s3$.
The decomposition $(M,L)=(M,L_1)\sharp (\s3 ,L_2)$ induces  the monomorphisms $j_1:H_1(M \smallsetminus L_1)\to H_1(M \smallsetminus L)$ and $j_2:H_1(\s3\smallsetminus L_2)\to H_1( M \smallsetminus L)$. Given $\sigma:\Z [H_1(M \smallsetminus L)]\to \C[G]$  induced by 
\hbox{$\sigma\in\hom(\textup{Tors}(H_1(L(p,q) \smallsetminus L)),\mathbb C^*)$,}  denote by  $\sigma_1$ and $\sigma_2$ its restrictions to $\C[j_1(H_1(M\smallsetminus L_1))]$  and $\C[j_2(H_1(\s3\smallsetminus L_2))]$, respectively. We have the following result.

\begin{theorem}
Let $L=L_1\sharp L_2\subset L(p,q)$, where  $L_2$ is local link.  
With the above notations it holds that $\Delta_L^{\sigma}=\Delta_{L_1}^{\sigma_1}\cdot  \Delta_{L_2}^{\sigma_2}$.
\end{theorem}
\begin{proof}
Since $(M,L)=(M,L_1)\sharp (\s3,L_2)$, by Seifert-Van Kampen theorem we have that  $\pi_1(M \smallsetminus L)$ has the generators of both $\pi_1(M\smallsetminus L_1,*)$ and $\pi_1(\s3\smallsetminus L_2, *)$, moreover, it has the relations of both of them, with the two additional conditions that joins the overpasses of the connected sum. So  the Alexander-Fox matrix  of $L$ is 
$$A_L=\left(\begin{array}{cccccc}\ &A_{L_1}&\ &\  &0 &\ \\\ & 0 \ &\ & \ &A_{L_2}&\ \\-1\ 0&\cdots& 0&1\ 0&\cdots&\ 0 \ \\0\ 1&\cdots& 0&0\ -1 &\cdots&\ 0 \end{array}\right),$$
where $A_{L_i}$ is the Alexander-Fox matrix of $L_i$, for $i=1,2$. If $d_k(A)$ denotes the greatest common division of all $k$-minors of a matrix $A$, then a simple computation shows that $d_{m+n-1}(A_L)=d_{n-1}(A_{L_1})\cdot d_{m-1}(A_{L_2})$. Therefore it is easy to see that $\Delta_L^{\sigma}=\Delta_{L_1}^{\sigma_1}\cdot  \Delta_{L_2}^{\sigma_2}$.
\end{proof}

\section{Example}\label{linkexample}

The following example is computed in part by hand for the group presentation starting from an arrow diagram and in part by computer.

\begin{example}\label{es1}

Consider the knot $K$ in the Seifert fibered space  $S(O, n, 1 | (1,2) )$ described by the diagram of Figure~\ref{fig:ex6}. It holds $\gamma=+1$ and $\delta=-1$.

\begin{figure}[htb]
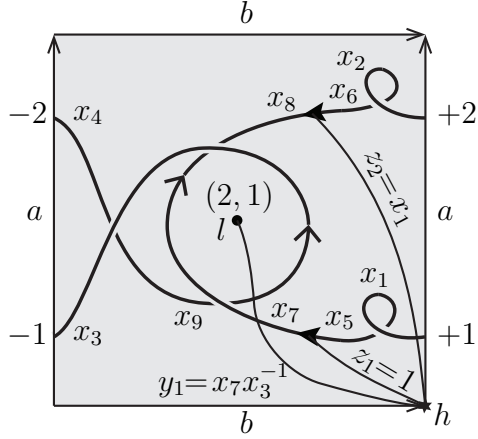

\centering
\begin{overpic}[page=45]{seif}
\put(101,50){$a$}\put(-5,50){$a$}  \put(50,-5){$b$}\put(50,101){$b$} 
\put(101,18){$+1$}\put(101,75){$+2$} \put(-10,18){$-1$}\put(-10,75){$-2$} 
\put(81,34){$x_1$}\put(75,90){$x_2$}\put(7,19){$x_3$}\put(7,76){$x_4$}
\put(72,23){$x_5$}\put(73,81.5){$x_6$}\put(58,25){$x_7$}\put(57,79.5){$x_8$}
\put(33,23){$x_9$}
\put(41,54){$(2,1)$}\put(44,45){$l$}
\put(29,6){$y_1\!\!=\!x_7x_3^{-1}$}\put(78,14.5){\rotatebox{-31}{$z_1\!\!=\!1$}}\put(81,67){\rotatebox{-62}{$z_2\!\!=\!x_1$}}
\put(100,-3){$h$}
\end{overpic}
\caption{Example of a knot in $S(O, n, 1 | (1,2) )$.}
\label{fig:ex6}
\end{figure}

According to Theorem \ref{OrGroup}, we wrote aside to each overpass the corresponding generator $x_{i}$.  The corresponding coefficients $\epsilon_{i}$, when necessary, are: $\epsilon_{1}=+1$, $\epsilon_{2}=-1$, $\epsilon_{3}=-1$, $\epsilon_{4}=+1$, $\epsilon_{5}=-1$, $\epsilon_{6}=+1$, $\epsilon_{7}=+1$ and $\epsilon_{8}=-1$.
The complete group presentation is the following one.
\begin{align*}
\pi_{1}&( M \smallsetminus K)= \langle \ x_{1}, x_{2}, \ldots, x_{8}, x_{9}, h, l, a, b \ | \\
&| \ x_{1}=x_{5}, x_{2}=x_{6}, x_{9}x_{7}x_{3}^{-1}x_{7}^{-1}=1, x_{7}x_{3}x_{8}^{-1}x_{3}^{-1}=1, x_{3}x_{4}x_{3}^{-1}x_{9}^{-1}=1, \\
&aba^{-1}b^{-1}l^{-2}x_{5}x_{6}^{-1}=1, x_{1}x_{2}^{-1}a h a^{-1}h^{-1}=1, b h b^{-1}h=1, \\
&l^{-2}= h^{-1}x_{7}x_{3}^{-1}l^{-2}x_{3}x_{7}^{-1}h, x_{7}=h^{-1} x_{5}h, x_{8}^{-1}=h^{-1}x_{1}x_{6}^{-1}x_{1}^{-1}h, \\
&x_{3}^{-1}= ab a^{-1}x_{1}^{-1} a b^{-1}a^{-1}, x_{4}= ab a^{-1}x_{2} a b^{-1}a^{-1}, l=x_{3}x_{7}^{-1}h \rangle
\end{align*} 

In order to compute the twisted Alexander polynomials, it is necessary to assign the correct power of the variable $z$ to each generator of the group. 
Using Equation \ref{homMnorFor}, the first homology group of the knot complement is $$H_{1}(M \smallsetminus K) = \langle \ g, l, h, a ,b \ | \ l^{-2}=1, h^{2}=1, l=h \ \rangle. $$
Hence $H_{1}(M \smallsetminus K) \cong \Z^{3} \oplus \Z_{2}$, where the free generators are $g, a$ and $b$, while $l$ and $h$ are sent to the generator of $\Z_{2}$.
In this case each generator $x_{i}$ of the group of the knot is sent to $g$, so the homomorphism $\tilde{\sigma}_{1} \colon \Z[\pi] \to \C[G]$, associated to the usual Alexander polynomial, sends $x_{i}$ to $z$, $a$ and $b$ again to $z$, while $l$ and $h$ are sent to $1$. Considering the twisted Alexander polynomial associated to the representation $\tilde{\sigma}_{-1}$, the free generators of the homology are sent to $z$ as before, while $l$ and $h$ are sent in $-1$. The number of the one-dimensional twisted Alexander polynomials we are going to find is the cardinality of the torsion part of the homology group of the knot, in this case it is two.

With these pieces of information, the computations shows that:

$$\Delta^{\sigma_{1}}_{K}=z^{4}-2z^{3}+2z-1$$
$$\Delta^{\sigma_{-1}}_{K}=z^{2} - 1$$

Using the coefficients of Lemma \ref{homologyclass}, the homology class $[K] \subset H_{1}(M) \cong \Z^{2} \oplus \Z_{2}$ is trivial. Theorem \ref{local} guarantees us that the knot $K$ is non-local, since its twisted Alexander polynomials are non-zero. 

\end{example}

\paragraph{acknowledgements}
The authors would like to thank Matija Cencelj for promoting their collaboration and Alessia Cattabriga for useful insights on twisted Alexander polynomials.

%
%




\newpage

\vspace{15 pt} {BO\v{S}TJAN GABROV\v{S}EK, Faculty of Mechanical Engineering, 
University of Ljubljana, SLOVENIA. E-mail: bostjan.gabrovsek@fs.uni-lj.si}

\vspace{15 pt} {ENRICO MANFREDI, Department of Mathematics,
University of Bologna, ITALY. E-mail: enrico.manfredi3@unibo.it}

\end{document}